\documentclass[12pt, reqno]{amsart}

\usepackage{amsmath,enumitem}
\usepackage[english,activeacute]{babel}
\usepackage{mathrsfs} 
\usepackage[latin1]{inputenc}
\usepackage{amssymb}
\usepackage{amsthm}
\usepackage{graphicx}
\usepackage{array}
\usepackage{pdflscape}
\usepackage{a4wide}
\usepackage{url}
\usepackage{float}
\usepackage{hyperref}
\usepackage{multirow}
\usepackage[capitalize,nameinlink,noabbrev,nosort]{cleveref}
\usepackage{txfonts}
\usepackage[dvipsnames]{xcolor}

\hypersetup{
    pdftoolbar=true,        % show Acrobat toolbar?
    pdfmenubar=true,        % show Acrobat menu?
    pdffitwindow=false,     % window fit to page when opened
    pdfstartview={FitH},    % fits the width of the page to the window
    pdftitle={},
    pdfauthor={},
    pdfsubject={},
    pdfkeywords={},
    pdfnewwindow=true,      % links in new window
    colorlinks=true,       % false: boxed links; true: colored links
    linkcolor=brown,          % color of internal links
    citecolor=brown,        % color of links to bibliography
    filecolor=brown,      % color of file links
    urlcolor=brown,           % color of external links
    unicode=true          % non-Latin characters in Acrobat's bookmarks
}

\theoremstyle{plain}
\newtheorem{theorem}{Theorem}

\theoremstyle{definition}

\newtheorem{remark}[theorem]{Remark}

\def\st#1#2{\begin{bmatrix}
#1 \\ #2 \end{bmatrix}}
\def\sts#1#2{\begin{Bmatrix}
#1 \\ #2 \end{Bmatrix}}

\newcommand{\N}{{\mathbb N}}

\newcommand{\Co}{{\mathbb C}}
\newcommand{\R}{{\mathbb R}}

\newcommand{\E}{{\mathbb E}}

\newcommand{\B}{{\mathcal B}}
\newcommand{\A}{{\mathcal A}}

\title
{$B$-Stirling numbers associated to potential polynomials}

\date{\today}

\author{Jos\'{e} A.~Adell}
\address{Departamento de M\'{e}todos Estad\'{i}sticos, Facultad de Ciencias, Universidad de Zaragoza, Spain}
\email{adell@unizar.es}

\author{Be\'ata B\'enyi}
\address{Be\'{a}ta B\'enyi, Faculty of Water Sciences, University of Public Service, Baja, Hungary}
\email{benyi.beata@uni-nke.hu}

\date{\today}
\begin{document}

\begin{abstract}
	We introduce the $B$-Stirling numbers of the first and second kind, which are the coefficients of the potential polynomials when we express them in terms of the monomials and the falling factorials, respectively. These numbers include, as particular cases, the partial and complete Bell polynomials, the degenerate and probabilistic Stirling numbers, and the $S$-restricted Stirling numbers, among others. Special attention is devoted to the computation of such numbers. On the one hand, a recursive formula is provided. On the other, we can compute Stirling numbers of one kind in terms of the other, with the help of the classical Stirling numbers. 
 	
\end{abstract}
	\maketitle

%%%%%%%%%%%%%%%%%%%%%%%

\section{Introduction}

%%%%%%%%%%%%%%%%%%%%%%

Let $\N$ be the set of positive integers and $\N_0=\N\cup\{0\}$. Throughout this paper, we assume that $j,k,m,n\in \N_0$, $x\in \R$, and $z\in \Co$ with $|z|\leq r$, for some  $r>0$, where $r$ may change from line to line.

Let $\A$ be the set of functions $A(z)$ which are analytic at $z=0$ and let $\B\subseteq \A$ be the set of functions $B(z)$ such that $B(0)=1$. If $A(z)\in \A$ and $B(z)\in \B$, we consider the sequence of polynomials $(P_n(x))_{n\geq 0}$ whose generating function is given by

\begin{align}\label{1}
	\sum_{n=0}^{\infty} {P}_n(x) \frac{z^n}{n!} = A(z)\left(B(z)\right)^x. 
\end{align}

If $A(0)\not=0$ and 
\begin{align*}
	B(z) = \exp \left(\sum_{i=1}^{\infty} H_i\frac{z^i}{i!}\right),\qquad H_1\not=0,
\end{align*}
then $(P_n(x))_{n\geq 0}$  is called a \emph{Sheffer polynomial sequence} (see Sheffer\cite{Sheffer1939}, Roman\cite{Roman1984}, Wang and Wang \cite{WangWang2009}, and Marcell\'{a}n et al. \cite{Marcellan2019} for equivalent definitions).

Given $B(z)\in \B$, the sequence $(P_n(B;x))_{n\geq 0}$ defined by 
\begin{align}\label{2}
	\sum_{n=0}^{\infty} P_n(B;x)\frac{z^n}{n!} =\left(B(z)\right)^{x},
\end{align}
is known in the literature as the sequence of \emph{potential polynomials} (see, for instance, Comtet \cite[Sec.3.5]{Comtet1974}, Howard \cite{Howard1982}, Cenkci \cite{Cenkci2009}, and Wang \cite{Wang2014}). If 
\begin{align*}
	A(z) = \sum_{n=0}^{\infty} A_n\frac{z^n}{n!},
\end{align*}
then we have from \eqref{1}, \eqref{2}, and \eqref{11} and \eqref{12} below 

\begin{align*}
	{P}_n(x) = \sum_{k=0}^{n}\binom{n}{k} A_kP_{n-k}(B;x).
\end{align*}
For this reason, we will restrict our attention to potential polynomials. To describe such polynomials, we introduce the sequences $(s_{B}(n,k))_{n\geq k}$ and $(S_{B}(n,k))_{n\geq k}$ respectively defined as
\begin{align}\label{3}
	\frac{(\log B(z))^k}{k!} = \sum_{n=k}^{\infty} s_B(n,k)\frac{z^n}{n!},
\end{align}
and 
\begin{align}\label{4}
	\frac{(B(z)-1)^k}{k!}= \sum_{n=k}^{\infty} S_B(n,k)\frac{z^n}{n!}.
\end{align}
We call such numbers the \emph{Stirling numbers of the first and second kind associated to $B(z)$}, respectively, or for short, the \emph{$B$-Stirling numbers.}

In first place, note that $B$-Stirling numbers extend the classical Stirling numbers of the first and second kind, respectively denoted by $s(n,k)$ and $S(n,k)$, in the following sense. Suppose that 
\begin{align}\label{5}
	I(z) = 1+z.
\end{align}
Then, we have from \eqref{3} and \eqref{4}
\begin{align}\label{6}
	s_I(n,k) = s(n,k) \qquad\mbox{and} \qquad S_I(n,k) = \delta_{n,k}. 
\end{align}
In this case, the potential polynomials are the falling factorials, that is,
\begin{align}\label{7}
	P_n(I;x) = (x)_n:= x(x-1)\cdots(x-n+1),\qquad ((x)_0=1).
\end{align}
Similarly, for \begin{align}\label{8}
	E(z) = e^z,
\end{align}
we have 
\begin{align}\label{9}
	s_E(n,k)=\delta_{n,k} \qquad\mbox{and}\qquad S_E(n,k) = S(n,k),
\end{align}
as well as 
\begin{align}\label{10}
	P_n(E;x) = x^n.
	\end{align}

In second place, it will be shown in \cref{section_4} that many sequences of known numbers are obtained as particular cases of $B$-Stirling numbers for an appropriate choice of the function $B(z)\in \B$. For instance, partial and complete Bell polynomials (cf. Comtet \cite{Comtet1974}, Yang \cite{Yang2008}, and Wang and Wang \cite{WangWang2009}), the degenerate Stirling numbers of both kinds introduced by Kim and Kim \cite{KimKim2020, KimKim2023}, the probabilistic Stirling numbers of both kinds introduced by the authors \cite{AdellLekuona2019, Adell2022, AdellBenyi2024}, and the $S$-restricted Stirling numbers considered by the second author \cite{BenyiRamirez2019,BenyiRamirezMendez2020,BenyiRamirezMendezWakhare2019}.

The paper is organized as follows. In \cref{section_2}, we show that potential polynomials are closely connected with $B$-Stirling numbers. Indeed, the $B$-Stirling numbers of the first kind (resp. second kind) are the coefficients of the potential polynomials when expressed in terms of the monomials $x^k$ (resp. the falling factorials $(x)_k$). As a consequence, the $B$-Stirling numbers of the first kind (resp. second kind) are obtained by differentiation (resp. computation of forward differences) at the origin of the potential polynomials. A recursive formula to compute both kinds of such numbers is also provided. 

In \cref{section_3}, we introduce two inner operations in the set $\B$ and describe the behaviour of the corresponding $B$-Stirling numbers and potential polynomials under these two operations. One important consequence, from a computational point of view, is the following. In dealing with particular examples, one kind of the $B$-Stirling numbers $s_B(n,k)$ or $S_B(n,k)$ is easier to compute than the other. Under such circumstances, we directly compute the easiest one, say $s_B(n,k)$, and then compute the other, say $S_B(n,k)$, in terms of $s_B(n,k)$ and the classical Stirling numbers.

%%%%%%%%%%%%%%%%%%%%%%%%%%%%%%%%%%%%%%%%%%%%%%%%

\section{$B$-Stirling numbers}\label{section_2}

%%%%%%%%%%%%%%%%%%%%%%%%%%%%%%%%%%%%%%%%%%%%%%%%

The $B$-Stirling numbers allow us to describe the potential polynomials (and also the Sheffer polynomials) in closed form, as shown in the following result. 

\begin{theorem}\label{theo_1}
	Let $B(z)\in \B$. Then,
	\begin{align*}
		P_n(B;x) = \sum_{k=0}^{n}s_B(n,k)x^k = \sum_{k=0}^{n} S_B(n,k)(x)_k.
	\end{align*}
\end{theorem}
\begin{proof}
	By \eqref{2} and \eqref{3} we have 
	\begin{align*}
		(B(z))^x&= \exp{(x\log B(z))}= \sum_{k=0}^{\infty}x^k\frac{(\log B(z))^k}{k!}\\
		&= \sum_{k=0}^{\infty}x^k\sum_{n=k}^{\infty} s_B(n,k)\frac{z^n}{n!} =\sum_{n=0}^{\infty}\frac{z^n}{n!} \sum_{k=0}^{n}s_B(n,k)x^k,
	\end{align*}
which shows the first identity. On the other hand, suppose that $|B(z)-1|<1$. By \eqref{2}, \eqref{4}, and the binomial expansion, we obtain
\begin{align*}
	(B(z)-1+1)^x &= \sum_{k=0}^{\infty} (x)_k \frac{(B(z)-1)^k}{k!} = \sum_{k=0}^{\infty} (x)_k\sum_{n=k}^{\infty} S_B(n,k)\frac{z^n}{n!}\\
	&= \sum_{n=0}^{\infty} \frac{z^n}{n!} \sum_{k=0}^{n} S_B(n,k)(x)_k,
\end{align*}
thus showing the second identity and completing the proof. 
\end{proof}
\begin{remark}\label{rem_2}
	Applying \cref{theo_1} to $I(z)=1+z$, we get from \eqref{6} and \eqref{7} the well known identity 
	\begin{align*}
		P_n(I;x)= (x)_n=\sum_{k=0}^{n} s(n,k)x^k.
	\end{align*}
Similarly, for $E(z) = e^z$, we have from \eqref{9} and \eqref{10}
\begin{align*}
	P_n(E;x) = x^n = \sum_{k=0}^{n} S(n,k) (x)_k.
\end{align*}
\end{remark}
Let $\mathcal{H}$ be the set of complex sequences $\mathbb{U}=(U_n)_{n\geq 0}$ having a finite generating function 
\begin{align*}
	\sum_{n=0}^{\infty} U_n\frac{z^n}{n!}.
\end{align*}
If $\mathbb{U},\mathbb{V}\in\mathcal{H}$, we define their binomial convolution $\mathbb{U}\times\mathbb{V}=((U\times V)_n)_{n\geq 0}$ as
\begin{align}\label{11}
	(U\times V)_n = \sum_{j=0}^{n}\binom{n}{j}U_jV_{n-j}.
\end{align}
It is well known (see, for instance, \cite{AdellLekuona2017}) that 
\begin{align}\label{12}
	\sum_{n=0}^{\infty}(U\times V)_n \frac{z^n}{n!} = \sum_{n=0}^{\infty} U_n\frac{z^n}{n!}\sum_{n=0}^{\infty} V_n\frac{z^n}{n!}.
\end{align}

The following result shows that $B$-Stirling numbers can be recursively computed. 

\begin{theorem}\label{theo_3}
	Let $B(z)\in \B$. For any $n\geq k\geq 1$, we have
	\begin{align}\label{15*}
		s_B(n,k) = \frac{1}{k} \sum_{j=k-1}^{n-1} s_B(j,k-1)s_B(n-j,1),
	\end{align}
and 
\begin{align}\label{16*}
	S_B(n,k) = \frac{1}{k} \sum_{j=k-1}^{n-1} \binom{n}{j}S_B(j,k-1)S_B(n-j,1).
\end{align}
In particular, 
\begin{align}\label{17*}
	s_B(k,k) =[s_B(1,1)]^k\qquad\mbox{and}\qquad S_B(k,k)= \left[S_B(1,1)\right]^k.
\end{align}
\end{theorem}

\begin{proof}
	Starting form the identity 
	\begin{align*}
		\frac{(\log B(z))^k}{k!}= \frac{1}{k}\frac{(\log B(z))^{k-1}}{(k-1)!}\log B(z),
	\end{align*}
formula \eqref{15*} follows from  \eqref{3}, \eqref{11}, and \eqref{12}, since

\begin{align}\label{18*}
	s_B(n,k) = S_B(n,k)=0,\quad n<k.
\end{align}
The proof of \eqref{16*} is similar starting from the identity
\begin{align*}
	\frac{(B(z)-1)^k}{k!} = \frac{1}{k}\frac{(B(z)-1)^{k-1}}{(k-1)!}(B(z)-1).
\end{align*}
Finally, we have from \eqref{15*}
\begin{align*}
	s_B(k,k) = s_B(k-1,k-1)s_B(1,1).
\end{align*}
Thus, the first equality in \eqref{17*} follows by induction on $k$. The second one is shown in a similar way.  
\end{proof}

Recall that the $r$th forward differences of a function $f:\R\rightarrow \R$ are recursively defined as 
\begin{align*}
	\Delta^{0}f(x) = f(x), \quad\Delta^1f(x) = f(x+1)-f(x), \quad \Delta^rf(z) = \Delta^1(\Delta^{r-1}f)(x), \quad r\in \N.
\end{align*}
Equivalently, 
\begin{align}\label{19}
	\Delta^rf(x) = \sum_{j=0}^{r} \binom{r}{j}  (-1)^{r-j}  f(x+j), \qquad r\in \N_0.
\end{align}
It is well known that if $p_n(x)$ is a polynomial of degree $n$, then 
\begin{align}\label{20}
	 \Delta^{r}p_n(x) = 0, \qquad r>n.
\end{align}
The $B$-Stirling numbers of the first and second kind are related, respectively, to the derivatives and the forward differences of the potential polynomials $P_n(B,x)$. In a certain sense, \cref{theo_4} is the inverse of \cref{theo_1}.

\begin{theorem}\label{theo_4}
	Let $B(z)\in\B$ and let $r\in \N_0$, $0\leq r\leq n$. Then, 
	\begin{align}\label{21}
		\frac{P_n^{(r)}(B;x)}{r!}=\sum_{k=r}^{n}\binom{k}{r} s_B(n,k)x^{k-r} = \sum_{j=r}^{n} \binom{n}{j} s_B(j,r)P_{n-j}(B;x),
	\end{align}
and
\begin{align}\label{22}
	\frac{\Delta^r P_n(B;x)}{r!}= \sum_{k=r}^{n}\binom{k}{r}S_B(n,k)(x)_{k-r}=\sum_{j=r}^{n}\binom{n}{j} S_B(j,r)P_{n-j}(B;x). 
\end{align}
In particular, 
\begin{align}\label{23}
	\frac{P_n^{(r)}(B;0)}{r!} = s_B(n,r),\qquad \mbox{and}\qquad \frac{\Delta^rP_n(B;0)}{r!} = S_B(n,r).
\end{align}
\end{theorem}
\begin{proof}
	The first equality in \eqref{21} readily follows from \cref{theo_1}. By differentiating $r$ times with respect to $x$ in \eqref{2} we obtain
	\begin{align*}
		\sum_{n=r}^{\infty} \frac{P_n^{(r)}(B;x)}{r!}\frac{z^n}{n!} = \frac{(\log B(z))^r}{r!}B(z)^x.
	\end{align*}
Hence, the second equality in \eqref{21} follows from \eqref{11}, \eqref{12}, and \eqref{18*}. 

On the other hand, it is easily checked by induction on $r$ that 
\begin{align*}
	\frac{1}{r!} \Delta^{r}(x)_k = \binom{k}{r}(x)_{k-r}, \qquad r=0,1,\ldots,k.
\end{align*}
Thus, the first equality in \eqref{22} follows by applying the operator ${\Delta^r}/{r!}$ to the second equality in \cref{theo_1} and taking into account \eqref{20}. Applying again the operator ${\Delta^r}/{r!}$ to both sides of \eqref{2} and recalling \eqref{20}, we get
\begin{align*}
	\sum_{n=r}^{\infty} \frac{\Delta^rP_n(B;x)}{r!}\frac{z^n}{n!} = \frac{(B(z)-1)^r}{r!}B(z)^x.	
\end{align*}
Therefore, the second equality in \eqref{22} follows from \eqref{11}, \eqref{12}, and \eqref{18*}. Finally, the two identities in \eqref{23} follow from \eqref{21} and \eqref{22}, respectively.
\end{proof}
\begin{remark}\label{rem_5}
	As in \cref{rem_2}, consider $E(z) = e^z$. In this case $P_n(E;x) = x^n$ and $S_E(n,k) = S(n,k)$, as noted in \eqref{9}. Hence, \eqref{19} and the second equality in \eqref{23} give us the well known formula 
	\begin{align*}
		\frac{1}{r!} \sum_{j=0}^{r}\binom{r}{j}(-1)^{r-j}j^n = S(n,r),\qquad r\in \N_0.
	\end{align*}
\end{remark}

%%%%%%%%%%%%%%%%%%%%%%%%%%%%%%%%%%%%%%%%%%%%%%%%%%%%%%%%%%%
\section{Inner operations in the set $\B$}\label{section_3}
%%%%%%%%%%%%%%%%%%%%%%%%%%%%%%%%%%%%%%%%%%%%%%%%%%%%%%%%%%%

Given $B(z),C(z)\in \B$, we consider the inner operation in $\B$ defined as 
\begin{align}\label{24}
	B\circ C(z) = B(C(z)-1).
\end{align} 

This operation is not commutative in general. The identity element is $I(z)=1+z$, i.e., 
\begin{align}\label{25}
	B\circ I(z) = I\circ B(z) = B(z), \qquad B(z)\in \B. 
\end{align}
The behaviour of the corresponding Stirling numbers with respect to this operation is described in the following result. 
\begin{theorem}\label{theo_6}
	Let $B(z),C(z)\in \B$. Then, 
	\begin{align}\label{26}
		s_{B\circ C}(n,k) = \sum_{j=k}^{n}S_C(n,j)s_B(j,k),
	\end{align}
and 
\begin{align}\label{27}
	S_{B\circ C} (n,k) = \sum_{j=k}^{n} S_C(n,j)S_B(j,k).
\end{align}
As a consequence, 
\begin{align}\label{28}
	s_C(n,k) = \sum_{j=k}^{n}S_C(n,j)s(j,k).
\end{align}
\end{theorem}
\begin{proof}
	By \eqref{3}, \eqref{4}, and \eqref{24}, we see that 
	\begin{align*}
		\frac{(\log B\circ C(z))^k}{k!}& = \sum_{j=k}^{\infty} s_B(j,k)\frac{(C(z)-1)^j}{j!} = \sum_{j=k}^{\infty}s_B(j,k)\sum_{n=j}^{\infty}S_C(n,j)\frac{z^n}{n!}\\
		&= \sum_{n=k}^{\infty}\frac{z^n}{n!}\sum_{j=k}^{n} S_C(n,j)s_B(j,k),
	\end{align*} 
thus showing \eqref{26}. Similarly, 
\begin{align*}
	\frac{(B\circ C(z)-1)^k}{k!} = \sum_{j=k}^{\infty} S_B(j,k)\frac{(C(z)-1)^j}{j!}=\sum_{n=k}^{\infty} \frac{z^n}{n!} \sum_{j=k}^{n} S_C(n,j)S_B(j,k),
\end{align*}
which shows \eqref{27}. Finally, \eqref{28} follows by choosing $B=I$ in \eqref{26} and recalling \eqref{6}. The proof is complete.
\end{proof}
\begin{remark}\label{rem_7}
	In matrix notation, \cref{theo_6} can be simply written as 
	\begin{align*}
		\mathrm{\textbf{s}}_{B\circ C} = \mathrm{\textbf{S}}_C \mathrm{\textbf{s}}_B, \quad \mathrm{\textbf{S}}_{B\circ C} = \mathrm{\textbf{S}}_C \mathrm{\textbf{S}}_B, \quad \mathrm{\textbf{s}}_C =\mathrm{\textbf{S}}_C\mathrm{\textbf{s}},
	\end{align*}
where $\mathrm{\textbf{s}}$ stands for the matrix $\mathrm{\textbf{s}} = (s(j,k))$, $j,k\in \N_0$.
\end{remark}
Another interesting inner operation in $\B$ is defined as 
\begin{align}\label{29}
	B\diamond C(z) = B(\log C(z)), \qquad B(z), C(z)\in \B. 
\end{align}
This non-commutative operation has identity element $E(z)=e^z$, that is,
\begin{align*}
	B\diamond E(z) = E\diamond B(z) = B(z), \qquad B(z)\in \B. 
\end{align*}
An analogous result to \cref{theo_6} for this inner operation is the following.
\begin{theorem}\label{theo_8}
	Let $B(z), C(z)\in \B$. Then,
	\begin{align}\label{30}
		s_{B\diamond C}(n,k) = \sum_{j=k}^{n} s_C(n,j)s_B(j,k),
	\end{align}
and 
\begin{align}\label{31}
	S_{B\diamond C}(n,k) = \sum_{j=k}^{n}s_C(n,j)S_B(j,k).
\end{align}
As a consequence,
\begin{align}\label{32}
	S_C(n,k) = \sum_{j=k}^{n} s_C(n,j)S(j,k).
\end{align}
\end{theorem}
\begin{proof}
	The proof is similar to that of \cref{theo_6} and therefore we omit it. 
\end{proof}
Let $B(z)\in \B$. Formulas \eqref{28} and \eqref{32} are important for computations. In fact, dealing with particular examples, one kind of the $B$-Stirling numbers $s_B(n,k)$ or $S_B(n,k)$ is easier to compute than the other. In such a case, the strategy consists of directly computing the easiest one and then use \eqref{28} or \eqref{32} to compute the other by using the classical Stirling numbers. 

\begin{remark}\label{rem_9}
	Choosing $C=E$ in \eqref{28} and $C=I$ in \eqref{32} and recalling \eqref{6} and \eqref{9}, we obtain the well known formula
	\begin{align*}
		\delta_{nk} = \sum_{j=k}^{n}S(n,j)s(j,k) = \sum_{j=k}^{n} s(n,j)S(j,k). 
	\end{align*}
	\end{remark}
The behaviour of the potential polynomials under the two operations considered above is specified in the following result.
\begin{theorem}\label{theo_10}
	Let $B(z), C(z)\in \B$. Then,
	\begin{align}\label{33}
		P_n(B\circ C;x) = \sum_{k=0}^{n}S_C(n,k)P_k(B;x),
	\end{align}
and 
\begin{align}\label{34}
	P_n(B\diamond C; x) = \sum_{k=0}^{n}s_C(n,k)P_k(B;x).
\end{align}
\end{theorem}
\begin{proof}
	By \eqref{2}, \eqref{4}, and \eqref{24}, we have 
	\begin{align*}
		(B\circ C(z))^{x} &= (B(C(z)-1))^x = \sum_{k=0}^{\infty} P_k(B;x)\frac{(C(z)-1)^k}{k!}\\
		&= \sum_{k=0}^{\infty} P_k(B;x) \sum_{n=k}^{\infty} S_C(n,k)\frac{z^n}{n!} = \sum_{n=0}^{\infty}\frac{z^n}{n!} \sum_{k=0}^{n}S_C(n,k)P_k(B;x),
	\end{align*}
thus showing \eqref{33}. In a similar way,
\begin{align*}
	(B\diamond C(z))^x &= (B(\log C(z)))^x = \sum_{k=0}^{\infty} P_k(B;x)\frac{(\log C(z))^k}{k!} \\
	&= \sum_{k=0}^{\infty} P_k(B;x)\sum_{n=k}^{\infty}s_C(n,k)\frac{z^n}{n!} =\sum_{n=0}^{\infty} \frac{z^n}{n!} \sum_{k=0}^{n}s_C(n,k)P_k(B;x),
\end{align*}
which shows \eqref{34} and completes the proof.
\end{proof}

%%%%%%%%%%%%%%%%%%%%%%%%%%%%%%%%%%%

\section{Examples}\label{section_4}

%%%%%%%%%%%%%%%%%%%%%%%%%%%%%%%%%%%

To illustrate the results of the previous sections, we consider the following examples.

%%%%%%%%%%%%%%%%%%%%%%%%%%%%%%%%%%%%%%%%%%%%%%%%%%%%%%%%%%%%%
\subsection{Example 1. Partial and complete Bell polynomials}
%%%%%%%%%%%%%%%%%%%%%%%%%%%%%%%%%%%%%%%%%%%%%%%%%%%%%%%%%%%%%

An important spacial case of the inner operation defined in \eqref{24} occurs when $B(z)=E(z)=e^z$. More precisely, we give the following result.

\begin{theorem}\label{theo_10*}
	Let $E\circ C(z)=\exp{(C(z)-1)}$, where $C(z)\in \B$. Then, 
	\begin{align}\label{A}
		s_{E\circ C}(n,k) = S_C(n,k), \qquad S_{E\circ C}(n,k) = \sum_{j=k}^n S_C(n,j)S(j,k).
	\end{align} 
In addition,
\begin{align}\label{B}
	P_n(E\circ C; x) = \sum_{k=0}^{n}S_C(n,k)x^k.
\end{align}
\end{theorem}
\begin{proof}
	The identities in \eqref{A} follow by choosing $B(z) = E(z)$ in \eqref{26} and \eqref{27} and taking into account \eqref{9}. Formula \eqref{B} follows by setting $B(z)= E(z)$ in \eqref{33} and applying \eqref{10}. The proof is complete. 
\end{proof}   

\cref{theo_10*} is a reformulation of the well known partial and complete Bell polynomials. Actually, recall that (see, for instance, Comtet \cite{Comtet1974}, or Wang and Wang \cite{WangWang2009}) the partial Bell polynomials are defined by 

\begin{align}\label{C}
	\frac{1}{k!} \left(\sum_{m=1}^{\infty}x_m\frac{z^m}{m!}\right)^k = \sum_{n=k}^{\infty} B_{n,k} (x_1,\ldots, x_{n-k+1})\frac{z^n}{n!},
\end{align}
whereas the complete Bell polynomials are given by 

\begin{align}\label{D}
	\exp \left(\sum_{m=1}^{\infty}x_m\frac{z^m}{m!}\right) = \sum_{n=0}^{\infty} B_n (x_1,\ldots, x_{n})\frac{z^n}{n!}.
\end{align}
Define $C(z)\in \B$ as 
\begin{align*}
	C(z) = 1+\sum_{m=1}^{\infty} x_m\frac{z^m}{m!}.
\end{align*}
Comparing the first identity in \eqref{A} with \eqref{C}, we see that 
\begin{align*}
	s_{E\circ C} (n,k) = S_C(n,k) = B_{n,k} (x_1,\ldots, x_{n-k+1}).
\end{align*}
Similarly, comparing \eqref{B} with \eqref{D}, we obtain the well known relation 
\begin{align*}
	B_n(x_1,\ldots, x_n) = P_n(E\circ C;1) = \sum_{k=0}^{n}S_C(n,k) = \sum_{k=0}^{n}B_{n,k}(x_1,\ldots, x_{n-k+1}).
\end{align*}
In other words, the partial Bell polynomials are the $E\circ C$-Stirling numbers of the first kind, whereas the complete Bell polynomials coincide with the potential polynomials associated to $E\circ C$ evaluated at $1$.

%%%%%%%%%%%%%%%%%%%%%%%%%%%%%%%%%%%%%%%%%%%%%%%%%%%
\subsection{Example 2. Degenerate Stirling numbers}
%%%%%%%%%%%%%%%%%%%%%%%%%%%%%%%%%%%%%%%%%%%%%%%%%%%

Let $\lambda\in \R\setminus\{0\}$. Denote by 

\begin{align}\label{35}
	(x)_{n,\lambda}= x(x-\lambda)\cdots (x-(n-1)\lambda) = \left(\frac{x}{\lambda}\right)_{n}\lambda^n, \quad n\in \N,\quad (x)_{0,\lambda}=1.
\end{align}

Recently, Kim and Kim \cite{KimKim2023} defined the degenerate Stirling numbers of the first and second kind as 
\begin{align}\label{36}
	(x)_n= \sum_{k=0}^{n}S_{1,\lambda}(n,k)(x)_{k,\lambda}\qquad\mbox{and}\qquad (x)_{n,\lambda} = \sum_{k=0}^{n}\sts{n}{k}_{\lambda}(x)_k,
\end{align}
respectively, and proved different results for these numbers. 

Let us describe such numbers using the tools developed above. In first place, we define the function $B_{\lambda}(z)\in \B$ as 
\begin{align}\label{37}
	B_{\lambda}(z) = (1+\lambda z)^{1/\lambda}.
\end{align}
As noted in \cite{KimKim2023}, $B_{\lambda}(z)\rightarrow E(z)=e^z$, as $\lambda\rightarrow 0$. In the following result, we compute the $B_{\lambda}$-Stirling numbers and the potential polynomials associated to $B_{\lambda}(z)$.

\begin{theorem}\label{theo_11}
	Let $B_{\lambda}(z)$ be as in \eqref{37}. Then,
	\begin{align}\label{38}
		s_{B_{\lambda}}(n,k)= \lambda^{n-k}s(n,k), 
	\end{align}
and 
\begin{align}\label{39}
	S_{B_{\lambda}}(n,k) = \sum_{j=k}^{n}\lambda^{n-j}s(n,j)S(j,k)=\frac{1}{k!}\sum_{j=0}^{k}\binom{k}{j}(-1)^{k-j}(j)_{n,k}.
\end{align}
In addition, 
\begin{align}\label{40}
	P_n(B_{\lambda};x) = (x)_{n,\lambda} = \sum_{k=0}^{n}\lambda^{n-k}s(n,k)x^{k} = \sum_{k=0}^{n}S_{B_{\lambda}}(n,k)(x)_k.
\end{align}
\end{theorem}

\begin{proof}
	From \eqref{37}, we directly compute $s_{B_{\lambda}}(n,k)$ as follows
	\begin{align*}
		\frac{(\log B_{\lambda}(z))^{k}}{k!}=\frac{1}{\lambda^k}\frac{(\log(1+\lambda z))^k}{k!} = \frac{1}{\lambda^k}\sum_{n=k}^{\infty} s(n,k) \frac{(\lambda z)^n}{n!},
	\end{align*}
which, by \eqref{3}, shows \eqref{38}. The first equality in \eqref{39} directly follows from \eqref{32} and \eqref{38}. 

On the other hand, we have from \eqref{35}, \eqref{37}, and the binomial expansion 
\begin{align*}
	(B_{\lambda}(z))^x = (1+\lambda z)^{x/\lambda}=\sum_{n=0}^{\infty} \left(\frac{x}{\lambda}\right)_n\lambda^n\frac{z^n}{n!}= \sum_{n=0}^{\infty} (x)_{n,\lambda}\frac{z^n}{n!},
\end{align*}
thus showing the first equality in \eqref{40}. The other two equalities in \eqref{40} follow from \cref{theo_1} and \eqref{38}.
Finally, we get from \eqref{19}, the second identity in \eqref{23}, and the first equality in \eqref{40}

\begin{align*}
	S_{B_{\lambda}}(n,k) = \frac{1}{k!} \Delta^k P_{n}(B_{\lambda};0) = \frac{1}{k!} \sum_{j=0}^{k} \binom{k}{j} (-1)^{k-j}(j)_{n,k}.
\end{align*}
This shows the second equality in \eqref{39} and completes the proof.
\end{proof}

Comparing the second identity in \eqref{36} with \eqref{40}, we see that 

\begin{align*}
	\sts{n}{k}_{\lambda} = S_{B_{\lambda}}(n,k).
\end{align*}
In other words, the degenerate Stirling numbers of the second kind coincide with the $B_{\lambda}$-Stirling numbers of the second kind. 

Define the function $C_{\lambda}(z)\in \B $ as
\begin{align}\label{41}
	C_{\lambda}(z) = \exp\left(\frac{1}{\lambda}((1+z)^{\lambda}-1)\right).
\end{align}
Note that $C_{\lambda}(z)\rightarrow I(z)=1+z$, as $\lambda\rightarrow 0$. The $C_{\lambda}$-Stirling numbers, as well as the potential polynomials $P_n(C_{\lambda};x)$ are not easy to compute in closed form. However, it is easily checked that the functions $B_{\lambda}(z)$ and $C_{\lambda}(z)$ satisfy the following relations

\begin{align}\label{42}
	C_{\lambda} \circ B_{\lambda}(z)= E(z) = e^z\qquad\mbox{and}\qquad 	B_{\lambda} \diamond C_{\lambda}(z)= I(z) = 1+z.
\end{align}

This allows us to relate the $C_{\lambda}$-Stirling numbers and the potential polynomials $P_n(C_{\lambda};x)$ with those associated to $B_{\lambda}(z)$.
We restrict our attention to the numbers $s_{C_{\lambda}}(n,k)$. 

\begin{theorem}\label{theo_12}
	Let $B_{\lambda}(z)$ and $C_{\lambda}(z)$ be as in \eqref{37} and \eqref{41}, respectively. Then, 
	\begin{align}\label{43}
		\delta_{n,k} = \sum_{j=k}^{n}S_{B_{\lambda}}(n,j)s_{C_{\lambda}}(j,k),
	\end{align}
\begin{align}\label{44}
	s(n,k) = \sum_{j=k}^{n} s_{C_{\lambda}}(n,j)\lambda^{j-k} s(j,k),
\end{align}
and 
\begin{align}\label{45}
	(x)_n= \sum_{k=0}^{n}s_{C_{\lambda}}(n,k) (x)_{k,\lambda}.
\end{align}
\end{theorem}

\begin{proof}
	Formula \eqref{43} follows from the first identity in \eqref{42}, by applying \eqref{9} and \eqref{26}. Formulas \eqref{44} and \eqref{45} follow from the second identity in \eqref{42}. Specifically, \eqref{44} follows from \eqref{6}, \eqref{30}, and \eqref{38}. To show \eqref{45}, we have from \eqref{7}, \eqref{34}, and the first equality in \eqref{40}
	
	\begin{align*}
		(x)_n = P_n(I;x) = P_n(B_{\lambda}\diamond C_{\lambda}; x) = \sum_{k=0}^{n}s_{C_{\lambda}}(n,k) P_{k}(B_{\lambda}; x)= \sum_{k=0}^{n}s_{C_{\lambda}}(n,k)(x)_{k,\lambda}.
	\end{align*}
This concludes the proof. 
\end{proof}

Comparing \eqref{36} with \eqref{45}, we see that 
\begin{align*}
 S_{1,\lambda} (n,k) = s_{C_{\lambda}}(n,k),
\end{align*}
that is, the degenerate Stirling numbers of the first kind coincide with the $C_{\lambda}$-Stirling numbers of the first kind. 

%%%%%%%%%%%%%%%%%%%%%%%%%%%%%%%%%%%%%%%%%%%%%%%%%%%%%%
\subsection{Example 3. Probabilistic Stirling numbers}
%%%%%%%%%%%%%%%%%%%%%%%%%%%%%%%%%%%%%%%%%%%%%%%%%%%%%%

Suppose that $B(z)\in \B$ is given by 
\begin{align}\label{46}
	B(z) = \E e^{zY},
\end{align}
where $\E$ stands for mathematical expectation and $Y$ is a complex-valued random variable having a finite moment generating function in a neighbourhood of the origin, i.e., 
\begin{align*}
	\E e^{r|Y|}<\infty, \quad \mbox{for some} \quad 0<r< \infty. 
\end{align*}
The authors introduced in \cite{AdellBenyi2024} (see also \cite{AdellLekuona2019}) the probabilistic Stirling numbers of the first and second kind, respectively denoted by $s_Y(n,k)$ and $S_Y(n,k)$, associated to $Y$ and defined as 
\begin{align}\label{47}
	s_Y(n,k) = (-1)^{n-k}s_B(n,k)\qquad \mbox{and}\qquad  S_Y(n,k)=S_B(n,k).
\end{align}
Explicit expressions of such numbers for different choices of $Y$, as well as applications in analytic number theory and probability theory can be found in \cite{Adell2022,AdellLekuona2019, AdellLekuona2021, LuoKimKim}.

Here, we focus our attention on the probabilistic interpretation of certain formulas given in the previous sections, whenever $B(z)$ has the form given in \eqref{46}. To this end, let $(Y_k)_{k\geq 1}$ be a sequence of independent copies of the random variable $Y$ and set 

\begin{align}\label{48}
	W_m = Y_1+\cdots +Y_m, \quad m\in \N\qquad (W_0=0).
\end{align}

In first place, the values of the potential polynomials $P_n(B;x)$ at nonnegative integers are in fact the moments of $W_m$. Specifically,
\begin{align}\label{49}
	P_n(B;m) = \E W_m^n.
\end{align}
Actually, it follows from \eqref{46} and \eqref{48} that 
\begin{align*}
	(B(z))^m = \left(\E e^{zY}\right)^m = \E e^{zW_m} = \sum_{n=0}^{\infty} \E W_m^n \frac{z^n}{n!},
\end{align*}
which, by virtue of \eqref{2}, shows \eqref{49}. 

In second place, \cref{theo_1}, \eqref{47}, and \eqref{49} give us 

\begin{align*}
	P_n(B;m) = \E W_m^n= \sum_{k=0}^{n}S_Y(n,k)(m)_k = \sum_{k=0}^{n}(-1)^{n-k}s_Y(n,k)m^k.
\end{align*}
In other words, the moments of $W_m$ can be computed in terms of the probabilistic Stirling numbers. Viceversa, we have from \eqref{19}, the second identity in \eqref{23}, and \eqref{49}

\begin{align*}
	S_Y(n,r) &= \frac{1}{r!}\Delta^r P_n(B;0) = \frac{1}{r!} \sum_{j=0}^{r} \binom{r}{j} (-1)^{r-j} P_n(B;j) \\
	&= \frac{1}{r!} \sum_{j=0}^{r} \binom{r}{j} (-1)^{r-j} \E W_j^n.
\end{align*}
This means that the probabilistic Stirling numbers of the second kind can be computed in terms of the moments of the random variables $W_j$.

%%%%%%%%%%%%%%%%%%%%%%%%%%%%%%%%%%%%%%%%%%%%%%%%%%%%%%%%%%%%%%%%%%%%%%%%%%%%%%%%%%%%
\subsection{Example 4. $S$-restricted Stirling numbers}
%%%%%%%%%%%%%%%%%%%%%%%%%%%%%%%%%%%%%%%%%%%%%%%%%%%%%%%%%%%%%%%%%%%%%%%%%%%%%%%%%%%%

\emph{$S$-restricted Stirling numbers of the first and second kind} (denoted by $\st{n}{k}_S$ and $\sts{n}{k}_S$, respectively) were introduced and studied in \cite{BenyiRamirez2019} and \cite{BenyiRamirezMendezWakhare2019}. Given a set $S\subseteq \N$, $\st{n}{k}_S$ is the number of permutations of $[n]=\{1,2,\ldots, n\}$ with $k$ cycles such that the size of each cycle is contained in $S$, whereas $\sts{n}{k}_S$ gives the number of partitions of the set $[n]$ having exactly $k$ blocks with the additional restriction that the size of each block is contained in $S$. 
If $S=\{k_1,k_2,\ldots\}$, the exponential generating functions are given by
\begin{align*}
\sum_{n=k}^{\infty} \st{n}{k}_S \frac{z^n}{n!} = 
\frac{1}{k!} \left(\sum_{i\geq 1} \frac{z^{k_i}}{k_i}\right)^k\qquad \mbox{and}\qquad	\sum_{n=k}^{\infty} \sts{n}{k}_S \frac{z^n}{n!} = 
	\frac{1}{k!} \left(\sum_{i\geq 1} \frac{z^{k_i}}{k_i!}\right)^k.
\end{align*}

Clearly, $S$-restricted Stirling numbers of the second kind are special cases of $B$-Stirling numbers of the second kind with 
\begin{align*}
	B^{(S)}(z) =  1+ \sum_{i\geq 1} \frac{z^{k_i}}{k_i!}.
	\end{align*}
For instance, by choosing $B(z)= \cosh z$, we obtain by the $B$-Stirling numbers of the second kind the number of set partitions of $[n]$ into $k$ non-empty blocks with \emph{even} sizes. 

The classes of set partitions with the restriction that a block can have \emph{at most} or \emph{at least} a certain number of elements have drawn the attention of many researchers. The counting sequences, the so called \emph{restricted Stirling numbers} and \emph{associated Stirling numbers} are well studied in the combinatorial literature (see, for instance, \cite{BenyiRamirez2019} for references).

Let us define the functions
\begin{align*}
	R^{(\leq m)}(z)= 1+\sum_{i=1}^m\frac{z^i}{i!} \qquad\mbox{and}\qquad R^{(\geq m)}(z)= 1+\sum_{i=m}^{\infty}\frac{z^i}{i!}.
\end{align*}
 Then $R^{(\leq m)}$-Stirling numbers coincide with the restricted Stirling numbers and  $R^{(\geq m)}$-Stirling numbers with the associated Stirling numbers (for further details, see \cite{BenyiRamirez2019}).

Similarly, let us define
\begin{align*}
	P^{(\leq m)}(z)= 1+\sum_{i=1}^m\frac{z^i}{i} \qquad\mbox{and}\qquad P^{(\geq m)}(z)= 1+\sum_{i=m}^{\infty}\frac{z^i}{i}.
\end{align*}
Then, $P^{(\leq m)}$-Stirling numbers of the second kind count permutations of $[n]$ into $k$ cycles such that each cycle is of length \emph{at most} $m$ and $P^{(\geq m)}$-Stirling numbers those with cycle-length \emph{at least} $m$. In particular, derangements (permutations without fix points) are counted by $P^{(\geq 2)}$-Stirling numbers of the second kind, the number of involutions is given by $B$-Stirling numbers of the second kind with  $B(z)= 1+z+z^2/2$ and $k$-pairings are enumerated by $B$-Stirling numbers with $B(z)= 1+z^2/2$. For further details and special cases see \cite{BenyiRamirezMendezWakhare2019} and \cite{Flajolet2009}.

Setting $B(z)=1/(1-z)$, $B$-Stirling numbers of the second kind coincide with Lah numbers. Lah numbers enumerate partitions of a set with $n$ elements into $k$ non-empty ordered lists. An introductory study of $S$-restricted Lah numbers, which are defined in the same spirit as the above mentioned $S$-restricted Stirling numbers, is given in \cite{BenyiRamirezMendez2020}. 

\section*{Acknowledgements} The first author is supported by Research Project DGA (E48$\_$23R).

\end{document}